\documentclass[11pt,leqno]{amsart}
\usepackage[english]{babel}
\usepackage{url}

\usepackage[normalem]{ulem}
\usepackage[letterpaper,top=2cm,bottom=2cm,left=3cm,right=3cm,marginparwidth=1.75cm]{geometry}

\usepackage{amsfonts,mathrsfs}
\usepackage{amssymb}
\usepackage{graphicx}
\usepackage{amsthm, amscd,indentfirst}
\usepackage{esint}
\usepackage{comment}      
\usepackage{bigints}

\usepackage{graphicx}
\usepackage{epstopdf}
\usepackage{latexsym}
\usepackage{color}
\usepackage[colorlinks=true, allcolors=blue]{hyperref}

\providecommand{\U}[1]{\protect\rule{.1in}{.1in}}
\newtheorem{theorem}{Theorem}[section]
\newtheorem*{acknowledgement*}{Acknowledgement}

\newtheorem{corollary}[theorem]{Corollary}

\newtheorem{definition}[theorem]{Definition}

\newtheorem{lemma}[theorem]{Lemma}

\newtheorem{proposition}[theorem]{Proposition}
\newtheorem{remark}[theorem]{Remark}

\title[Rigidity and non-existence results for collapsed translators]{Rigidity and non-existence results \\for collapsed translators}

\author[D. Impera]{Debora Impera}
\address[Debora Impera]{Dipartimento di Scienze Matematiche "Giuseppe Luigi Lagrange", Politecnico di Torino, Corso Duca degli Abruzzi, 24, Torino, Italy, I-10129}
\email{debora.impera@polito.it}

\author[N.M. M{\o}ller]{Niels Martin M{\o}ller}
\address[Niels Martin M{\o}ller]{
  Copenhagen Center for Geometry and Topology (GeoTop), Department of Mathematical Sciences,
  University of Copenhagen,
  DK-2100 Copenhagen, Denmark.
}
\email{nmoller@math.ku.dk}

\author[M. Rimoldi]{Michele Rimoldi}
\address[Michele Rimoldi]{Dipartimento di Scienze Matematiche "Giuseppe Luigi Lagrange", Politecnico di Torino, Corso Duca degli Abruzzi, 24, Torino, Italy, I-10129}
\email{michele.rimoldi@polito.it}

\thanks{N.M. M\o{}ller was partially supported by DFF Sapere Aude 7027-00110B, Carlsberg Foundation CF21-0680 and Danmarks Grundforskningsfond CPH-GEOTOP-DNRF151. D. Impera and M. Rimoldi were both partially supported by INdAM-GNSAGA, and they acknowledge partial support by PRIN 2017 Real and Complex Manifolds: Topology, Geometry and holomorphic dynamics 2017JZ2SW5, and PRIN 2022 Real and Complex Manifolds: Geometry and Holomorphic Dynamics 2022AP8HZ9.}



\begin{document}
\maketitle

\begin{abstract}
We prove a rigidity result for mean curvature self-translating solitons, characterizing the grim reaper cylinder as the only finite entropy self-translating 2-surface in $\mathbb{R}^3$ of width $\pi$ and bounded from below. The proof makes use of parabolicity in a weighted setting applied to a suitable universally $L$-superharmonic function defined on translators in such slabs.
\end{abstract}

\section{Introduction}
The mean curvature flow has been an active field since the 1970s (for some key references, see e.g. \cite{T78}, \cite{B78}, \cite{H84}, \cite{GH86}, \cite{G87}, \cite{H95}, \cite{W02}, \cite{CM11}, \cite{CM12}). Much of the work has been focused on singularity models for the flow, i.e. the self-similar solitons.  

The oldest known nontrivial complete embedded soliton is Calabi's self-translating curve $\Gamma$ in $\mathbb{R}^2$ (see Grayson \cite{G87}), also sometimes referred to as the ``grim reaper'', which appeared explicitly as early as 1956 in the paper \cite{M56}. This special solution can be seen as somewhat analogous to Hamilton's cigar soliton \cite{H88} in the Ricci flow, where G. Perelman's ``no cigar'' theorem \cite{perelman} is a central part of the theory.


Self-translating solitons arise naturally in the study of ``Type II'' singularities in the mean curvature flow. Huisken and Sinestrari \cite{HS99a} showed, based on Hamilton's results in \cite{H95}, that blow-up limit flows at Type II singularities of mean convex mean curvature flows are complete, self-translating solitons of the form $\Sigma^k \times \mathbb{R}^{n-k}$, where $\Sigma^k$ denotes a convex translator in $\mathbb{R}^{k+1}$, with $k = 1, \dots, n$. For example, a so-called grim reaper cylinder (aka. grim plane)
\begin{equation}\label{def:grim}
\Gamma \times \mathbb{R} = \{\left(x_1,x_2,-\log(\sin x_1)\right):\: x_1\in (0,\pi)\}.
\end{equation}
We note that, in some situations, such a reaper cylinder can then be ruled out as a possible singularity model (for more about the mean convex case see also e.g. \cite{HS99b}, \cite{Wh00}, \cite{Wh03}, \cite{SW09}, \cite{A12}, and \cite{HK17}).

In this paper, we normalize so that a translator $\Sigma^{2}\subseteq \mathbb{R}^{3}$ means a smooth, complete connected surface in $\mathbb{R}^{3}$, which evolves according to the mean curvature flow as unit speed translation in the direction of the positive $x_{3}$-axis. Therefore $\Sigma^{2}$ satisfies the following equation
\begin{equation}\label{Eq_Transl}
    \mathbf{H}=e_{3}^{\bot},
\end{equation}
where $(\cdot)^{\bot}$ denotes the projection on the normal bundle of $\Sigma$. Note that here we are using the conventions (which differ up to signs from what some other authors use) that
\[
 \mathbf{H}=H\nu=\mathrm{tr}(\mathbf{A}),
\]
where $\nu$ is the unit normal vector of $\Sigma$, $H$ is the mean curvature function, and the second fundamental form $\mathbf{A}$ of the immersion is defined as the generalized Hessian of the immersion.
When possible, we will choose the unit normal $\nu$ to be inward pointing, and we will say that $\Sigma$ is mean convex if it holds that $H\geq0$ (as e.g. for the round sphere). See Section \ref{Sec_Examples} for some examples, with references.

J. Spruck and L. Xiao proved in \cite{SX} that every complete mean convex translator in $\mathbb{R}^{3}$ is actually convex. X. J. Wang proved in \cite{W} that any convex translator in $\mathbb{R}^{n+1}$ which is not an entire graph lies in a vertical slab of finite width (i.e. is collapsed). Hence, using the grim reaper cylinder as a barrier, the following result has already been shown.

\begin{proposition}[\cite{HIMW}, \cite{BLT20}]\label{Prop_RigGrimGraph}
Let $\Sigma^2\subseteq\mathbb{R}^{3}$ be a translator, which is a graph $\{(x_1,x_2,u(x_1,x_2))\}$ and contained in a vertical slab of width $\pi$. Then, up to rigid motions, $\Sigma$ is a standard grim reaper cylinder $\Gamma\times\mathbb{R}$.
\end{proposition}

Namely, D. Hoffman, T. Ilmanen, F. Mart\'in, and B. White have actually obtained in \cite{HIMW} the full classification of complete graphical (into the translation direction, equivalently mean convex) translators in $\mathbb{R}^{3}$. These are either tilted grim reaper cylinders, $\Delta$-wings, or the bowl soliton (the unique entire graph over $\mathbb{R}^2$). Both the (non-trivially) tilted grim reaper cylinders and the $\Delta$-wings lie in vertical slab regions of widths strictly greater than $\pi$.
\medskip

In this paper we do not assume such graphicality (i.e. mean convexity) of the translators. This more general setting opens the possibility for a much wider variety of examples. For an overview of known examples of translators which are confined to some vertical slab, we refer to Section \ref{Sec_Examples} below.
\medskip

Our main result is the following rigidity theorem for the grim reaper cylinder among relatively low-width translators confined in possibly slanted upper halfspaces.

\begin{theorem}\label{Thm_RigB}
Let $\Sigma^2\subseteq \mathbb{R}^{3}$ be a translator with finite entropy contained in a vertical slab of width $\lambda\pi$ for some $\lambda\geq 1$. If $\Sigma\subseteq\left\{\Big(1-\frac{\lambda^2}{2}\Big)x_3+\alpha x_2\geq c\right\}$ for some $\alpha^2\geq\lambda^2-1$ and $c\in\mathbb{R}$, then necessarily $\lambda=1$, $\alpha=0$, and $\Sigma$ is a grim reaper cylinder $\Gamma\times\mathbb{R}$.
\end{theorem}

\begin{remark}
\rm{
The geometrically most natural case of Theorem \ref{Thm_RigB} corresponds to the parameters $\lambda=1$, $\alpha=0$, however we state a formally more general version, which will follow from the proof.

Note that the tilted grim reaper cylinders are actually contained within certain slanted upper halfspaces. However, these halfspaces have different slopes compared to those in the assumptions of Theorem \ref{Thm_RigB}. Moreover, note also that the $\Delta$-wings (see Section \ref{Sec_Examples}) are examples of translators which are  contained in an upper halfspace, but they violate the width condition.}   
\end{remark}

We emphasize in particular the following rigidity result which is a direct consequence of Theorem \ref{Thm_RigB}.

\begin{corollary}\label{Thm_RigA}
Let $\Sigma^2\subseteq \mathbb{R}^{3}$ be a translator with finite entropy and contained in a vertical slab of width $\omega = \pi$. Suppose that $\Sigma$ is contained in an upper halfspace. Then, up to rigid motions, $\Sigma$ is the standard grim reaper cylinder $\Gamma\times\mathbb{R}$.
\end{corollary}

\begin{remark}
\rm{
Concerning the width bound in Corollary \ref{Thm_RigA}, note that the so-called $\Delta$-wings (see \cite{BLT20} and \cite{HIMW}) give examples of translators of finite entropy which are bounded from below and contained in vertical slabs of width $w>\pi$. This shows the necessity/sharpness of the condition $w = \pi$ for the uniqueness in Corollary \ref{Thm_RigA}.

To the authors' knowledge, it is still an open problem whether there could exist examples of non-planar translators in a slab of width less than $\pi$, which are not bounded from below. 
By Proposition \ref{Prop_RigGrimGraph}, such examples could certainly not be graphs. However one might naively imagine connecting two planes by a catenoidal neck and adding a (perhaps large) perturbation; see also \cite[Remark 1.7]{HMW_Spruck}.

 Note also that with controlled topology and prescribing furthermore the asymptotic behavior outside of a suitable cylinder, a rigidity result for the grim reaper cylinder has been proven in \cite{MPSS}.
}
\end{remark}

In some previous work (e.g. \cite{C20}, \cite{GMM22}), one relied on classifying by reducing from more general surfaces down to the (semi)graphical i.e mean convex case, and then referring to the full classification of such by \cite{HIMW}, \cite{HMW22}. However, here, our proof of the main Theorem \ref{Thm_RigB} is direct and does not rely on the classification of mean convex solitons.

The proof will instead be in the spirit of the proof of the classical halfspace theorem for 2-dimensional properly immersed minimal surfaces in 3-space, using directly the definition of parabolicity (unlike the original proof in \cite{HM90} which had used nonlinear barriers). Namely, as a consequence of Theorem 3.1 in \cite{KCMR} when $\partial\Sigma =\emptyset$, classical minimal 2-surfaces (with everywhere $H=0$), if contained in a halfspace, are parabolic for the surface Laplacian $\Delta$. Recall that in this case one can simply use the universal harmonicity w.r.t $\Delta$ of the ambient coordinates of $\mathbb{R}^3$ on the minimal surface, to see that if the surface is contained in e.g. the set $x_3 > 0$, then in fact the parabolicity implies $x_3 = \mathrm{const}$.

In our case, we similarly work with a function whose level sets are grim reaper cylinders (in fact slightly more general functions and conclusions). Furthermore, we have to find the correct Laplacian for which there's simultaneously both parabolicity and universal superharmonicity.
In the minimal surface case, just discussed, one discarded the second fundamental form term from the stability operator $\Delta + |A^2|$,
meaning that in the translator case, where the stability operator is $\Delta + e_3\cdot\nabla + |A^2|$, it would seem natural to consider $\Delta + e_3\cdot\nabla$. Unfortunately, parabolicity for this operator on translators fails, due to the exponentially increasing volumes of the metric $e^{x_{3}}\delta_{ij}$ in the translation direction. However, flipping a sign on the gradient term to get $L\doteq\Delta - e_3\cdot\nabla$, which we will below be viewing as a Laplacian on a weighted manifold, and imposing the natural lower boundedness condition on the translator, finally yields a situation where both the crucial properties are now satisfied:
every lower-bounded translator $\Sigma^2$ of finite entropy is $L$-parabolic, and the grim reaper function is universally $L$-superharmonic.
\medskip

As another consequence of Theorem \ref{Thm_RigB} we obtain in particular the following non-existence result.

\begin{corollary}\label{Thm_Nonex}
There do not exist any $2$-dimensional translators $\Sigma^2\subseteq\mathbb{R}^{3}$ with finite entropy, contained in a slab of width $\omega$ and satisfying one of the following assumptions:
\begin{itemize}
\item[(a)] $\omega<\pi$, and $\Sigma$ is contained in an upper halfspace w.r.t. the translating direction;
\item[(b)] $\omega=\pi$, and $\Sigma$ is contained in a slanted halfspace of the form $\{x_3+\beta x_2\geq c\}$, with $\beta\neq 0$.
\end{itemize} 
\end{corollary}

Note that a similar result as the first part of the previous Corollary \ref{Thm_Nonex}, was obtained in \cite[Proposition 9.1]{GMM22}. In this latter paper the same conclusion was reached imposing, additionally to the finite entropy requirement, topological assumptions (finite genus and only one end, in the sense of connectedness at infinity) instead of the confinement to an upper halfspace, and with different techniques. These topological properties are features of stable translators. See \cite{KS, IR_MathZ} and Section \ref{SubSect_Stab}.

In all previous rigidity or non-existence results for collapsed translator some type of boundedness assumption was needed in order to reach of conclusion. When we have a precise description of an $x_3$-slice of the translator we can actually avoid this extra assumption. Indeed our techniques also permit obtaining the following:

\begin{theorem}\label{Thm_RigPrescrSlice}
Let $\Sigma^2\subseteq\mathbb{R}^{3}$ be a translator with finite entropy contained in a vertical slab. Assume that there exists a regular value $a\in\mathbb{R}$ of $x_3$ such that $\Sigma\cap \left\{x_{3}=a\right\}$ is a straight line parallel to the slab and that $\Sigma\cap \left\{x_{3}< a\right\}\neq\emptyset$. Then $\Sigma$ must be a vertical plane parallel to the slab.
\end{theorem}

We note that the assumption of collapsedness (i.e. the translators being contained in vertical slabs) is a very natural one, for example in the light of the classification of the projected convex hulls of translators in \cite{CM19}, namely as: slabs (incl. their degenerate cases of flat planes, halfspaces or the entire space).

\medskip

The proof of Theorem \ref{Thm_RigB} will be presented in Section \ref{Section_RigSlabs}, while the proof of Theorem \ref{Thm_RigPrescrSlice} will be discussed in Section \ref{Section_RigPrescrBnd}.

We additionally include Appendix \ref{Equidist}, of independent interest, which contains the characterization of equidistant surfaces to a translator with respect to the ambient Ilmanen conformal metric. This yields an interesting foliation of the vertical $\pi$-slab, which one might expect could be geometrically significant. Foliations of this kind were used e.g. in \cite{Mazet} for proving halfspace type theorem in other settings. We however also note that it gives a different foliation from the level sets of the ambient function in \eqref{universal_super} which were eventually used in the proof of our main results, Theorem \ref{Thm_RigB}, Corollary \ref{Thm_RigA} and Theorem \ref{Thm:RigPar}.
\medskip

\begin{center}
\textit{Acknowledgements}
\end{center}

We would like to thank the anonymous referees for their thoughtful and constructive comments, which significantly improved the quality of this manuscript.

\medskip

\begin{center}
\textit{Added note}
\end{center}

During the writing of the manuscript, the authors were made aware that some results in the same direction as Theorem \ref{Thm_RigB} were also obtained in the very recent preprint by D. Hoffman, F. Mart\'in and B. White, \cite{HMW_Spruck}. The two proofs are however clearly different in spirit and the results as well have different features. The most evident advantage of the approach in \cite{HMW_Spruck}, using suitable nonlinear barriers, is that they only assume that the slab-confined translator is lower bounded on one $x_2$-slice, while our upper halfspace assumption is asking that there are no wings going down. 

On the other hand, our direct approach permits easily getting a sharp rigidity statement also in the boundary-less setting, which seems to be more difficult to handle with the techniques in \cite{HMW_Spruck}, where the translators with boundary in their Theorem 1.3 are assumed to obey the Dirichlet-type condition that the boundary consists of straight lines. Also, while in their Theorem 1.4 there is no rigidity or uniqueness statement, we do obtain one in our Theorem \ref{Thm_RigB}. Furthermore, we explicitly note that, when considering slanted halfspaces, our Theorem \ref{Thm_RigB} even gives non-existence also in slightly wider slabs, up to the width $\sqrt{2}\pi$.

\section{Examples}\label{Sec_Examples}

We list here all the examples of complete translators, i.e. solutions to \eqref{Eq_Transl}, in slabs of $\mathbb{R}^3$ which we are aware of at the moment.

\begin{enumerate}
 \item Vertical planes.
 \item Grim reaper cylinders $\Gamma\times \mathbb{R}$, where $\Gamma$ is the grim reaper curve (see \eqref{def:grim}), as well as the so-called tilted grim reaper cylinders (see e.g. \cite{SX}). 
 \item The so-called $\Delta$-wings \cite{BLT20}, \cite{HIMW}.
 \item Nguyen's periodic examples aka. ``tridents'' \cite{N09}, \cite{HMW22}.
 \item Pitchforks (by Hoffman, Mart\'in and White \cite{HMW_Scherk}).
 \item Prongs (by Hoffman, Mart\'in and White, \cite{HMW_Annuli_InPrep}).
 \item Capped and uncapped translating annuloids (by Hoffman, Mart\'in and White \cite{HMW_Annuli_InPrep}, \cite{HMW_Spruck}).
\end{enumerate}

Among these examples: note that the capped translating annuloids are collapsed and bounded from below. In particular, our Theorem \ref{Thm_RigB} shows that the width of any example with those two properties has at least the width of the grim reaper cylinder.
\medskip

Apart from these, there are of course numerous examples of translators not contained in slabs, e.g. the bowl soliton, the Clutterbuck-Schn\"urer-Schulze examples, \cite{CSS}, and gluings of these, \cite{DDPN}.

\section{Some preliminary material}

\subsection{Background on weighted manifolds and parabolicity.}
Let $(M^m, g_{M})$ be a complete $m$-dimensional Riemannian manifold with (possibly empty) smooth boundary $\partial M$. Given $h\in C^{\infty}(M)$ we can consider the weighted manifold $M^{m}_h=(M^{m}, g_{M}, e^{-h}d\mathrm{vol}_M)$, 
where $d\mathrm{vol}_M$ denotes the canonical Riemannian volume form on 
$M$. Associated to the weighted manifold $M_h$ 
there is a natural divergence form second order diffusion operator, the $h$-Laplacian,
defined on $u$ by
\begin{equation}\label{Def:f-Laplacian}
\Delta_{h}u=e^{h}\operatorname{div}\left(e^{-h}\nabla u\right)=
\Delta u -g_{M}( \nabla u, \nabla h).
\end{equation}
This is clearly symmetric on the space $L^{2}(M_h)$, endowed with the inner product
\[
(u,v)_{L^2(M_h)}=\int_{M} uv e^{-h}d\mathrm{vol}_{M}.
\]
Let $\eta$ be the outward pointing unit normal to $\partial M$. 

\begin{definition}\label{def_NeumannParabolicity}
The weighted manifold $M_h$ is said to be (Neumann) $h$-\textit{parabolic} if for every $u\in C^{0}(M)\cap W^{1,2}_{\mathrm{loc}}(\mathrm{int}(M), e^{-h}d\mathrm{vol}_{M})$ such that 
\[
\begin{cases}
\Delta_{h}u\geq0&\mathrm{in}\,\,\mathrm{int}(M)\\
\frac{\partial u}{\partial \eta}\leq 0 &\mathrm{on}\,\,\partial M\\
\sup_{M}u<+\infty&
\end{cases}
\]
$u$ is necessarily constant.    
\end{definition}

In the boundary-less setting we just leave out the boundary condition. Moreover, when $h\equiv0$ we will simply say that $M$ is (Neumann) parabolic.
\medskip

As a matter of fact, $h$-parabolicity is related to a wide class of equivalent properties involving the recurrence of the Brownian motion, $h$-capacities of condensers, the heat kernel associated to the drifted Laplacian, weighted volume growth, 
function theoretic tests, global divergence theorems
and many other geometric and potential-analytic properties. Referring to \cite{ILPS} and \cite{IPS} for a general presentation of the theory and for detailed proofs, here we limit ourselves to pointing out a few characterizations.

\begin{proposition}[\cite{IPS}, \cite{ILPS}]\label{PropAhlforsStokes}
Let $M_h$ be a weighted manifold with boundary $\partial M\neq \emptyset$ and outward pointing unit normal $\eta$. Then $M_h$ is (Neumann) $h$-parabolic if and only if one of the following equivalent conditions holds:
\begin{itemize}
\item[(a)] (Ahlfors maximum principle) For every domain $D\subseteq M$ with $\partial_0D\doteq\partial D\cap\mathrm{int}(M)\neq\emptyset$ and for every $u\in C^{0}(\bar{D})\cap W^{1,2}_{\mathrm{loc}}(D, e^{-h}d\mathrm{vol}_{M})$ satisfying
\[
\begin{cases}
\Delta_{h}u\geq0&\mathrm{in}\,\,\mathrm{int}(D)\\
\frac{\partial u}{\partial \eta}\leq 0 &\mathrm{on}\,\,\partial M\cap D\\
\sup_{D}u<+\infty&
\end{cases}
\]
in the weak sense, it holds that
\[
\sup_D u=\sup_{\partial_0 D}u.
\]
\item[(b)] (Global $L^2$-Stokes theorem) Let $X$ be a vector field on $M$ satisfying 
\begin{align*}
&|X|\in L^2(M_h),\\
&g_{M}( X,\eta)\in L^{1}((\partial M)_h),\\
&\mathrm{div}X\in L^{1}_{\mathrm{loc}}(M), \qquad(\mathrm{div}X)_{-}\in L^{1}(M_h).
\end{align*}
Then
\[
\int_{M}e^h\mathrm{div}\left(e^{-h}X\right)e^{-h}d\mathrm{vol}_{M}=\int_{\partial M}g_{M}( X,\eta) e^{-h}d\mathrm{vol}_{m-1},
\]
where $d\mathrm{vol}_{m-1}$ is the $(m-1)$-dimensional Hausdorff measure.
\item[(c)] For any compact set $K\subset M$, the $h$-capacity 
\[
\ \mathrm{Cap}_{h}(K,M)=\inf\left\{\int_{M}|\nabla u|^2e^{-h}d\mathrm{vol}_{M}:\,u\in C_{c}^{\infty}(M),\, u\geq 1\,\mathrm{on}\,K\right\}
\]
is zero.
\end{itemize}
Moreover, when $D=M$, if $M_h$ is (Neumann) $h$-parabolic and $u\in C^{0}(M)\cap W^{1,2}_{\mathrm{loc}}(D, e^{-h}d\mathrm{vol}_{M})$ satisfies
\[
\begin{cases}
\Delta_{h}u\geq0&\mathrm{in}\,\,\mathrm{int}(M)\\
\sup_{M}u<+\infty&
\end{cases}
\]
then
\[
\sup_M u=\sup_{\partial M}u.
\]
\end{proposition}

\begin{lemma}\label{LemInfhPar}
Let $(M, g_{M})$ be a (Neumann) parabolic Riemannian manifold with boundary $\partial M\neq \emptyset$. Then, for any $h\in C^{\infty}(M)$ with $\inf_{M} h>-\infty$, $(M, g_{M})$ is also (Neumann) $h$-parabolic. 
\end{lemma}
\begin{proof}
Take any vector field $X$ on $M$ satisfying 
\begin{align*}
&|X|\in L^2(M_h),\\
&g_{M}( X,\eta)\in L^{1}((\partial M)_{h}),\\
&\mathrm{div}X\in L^{1}_{\mathrm{loc}}(M), \qquad(\mathrm{div}X)_{-}\in L^{1}(M_h),
\end{align*}
where $\eta$ is the outward pointing unit normal to the boundary $\partial M$. Since $\inf_{M}h>-\infty$ one easily get that that 
\begin{align*}
&|e^{-h}X|\in L^2(M),\\
&g_{M}(e^{-h}X,\eta)\in L^{1}(\partial M),\\
&(\mathrm{div}(e^{-h}X))_{-}\in L^{1}(M).
\end{align*}
By Proposition \ref{PropAhlforsStokes} we hence get that
\begin{align*}
\int_{M}e^h\mathrm{div}\left(e^{-h}X\right)e^{-h}d\mathrm{vol}_{M}=&\int_{M}\mathrm{div}(e^{-h}X)d\mathrm{vol}\\=&\int_{\partial M}g_{M}(e^{-h}X, \eta)d\mathrm{vol}=\int_{\partial M}g_{M}(X, \eta)e^{-h}d\mathrm{vol},
\end{align*}
and hence, by Proposition \ref{PropAhlforsStokes} $(M, g_{M})$ is $h$-parabolic.
\end{proof}

From a geometric point of view, $h$-parabolicity is related to the growth-rate of the weighted volumes of intrinsic metric objects. Indeed, exploiting a standard result due to Grigor'yan, \cite[Theorem 7.3]{Gr}, in \cite{ILPS} is given a direct proof of the following result, which will be sufficient for our scope here.
\begin{proposition}[see e.g. \cite{ILPS}]\label{Prop_quadparabolicity}
Let $M_h$ be a complete weighted manifold with (possibly empty) boundary $\partial M$. If, for some reference point $o\in M$,
\begin{equation*}
\mathrm{vol}_h(B_{r}^M(o))= \mathcal{O}(r^2),\qquad r\to+\infty
\end{equation*}
then $M_h$ is $h$-parabolic.
\end{proposition}
Here we are denoting by $\mathrm{vol}_h\left(B_{r}^M(o)\right)$ the weighted volume of the geodesic ball $B_{r}^M(o)$ with respect to $g_M$ of radius $r$ centered at a reference point 
$o\in M$.

\subsection{Useful facts about translators} In this subsection we will fix some notation and collect some general results about translators that will be used throughout the paper.

First of all, let $x:\Sigma^2\to\mathbb{R}^{3}$ be a translator of the mean curvature flow i.e. satisfying \eqref{Eq_Transl}. Hereafter, $(x_{1},x_{2},x_{3})$ are the standard coordinates of $\mathbb{R}^{3}$ and $(e_{1}, e_{2}, e_{3})$ is the standard orthonormal basis of $\mathbb{R}^{3}$. We will write the standard Euclidean metric on $\mathbb{R}^{3}$ as $g_{\mathbb{R}^{3}}=\delta_{ij}$, but we will also often use the more standard notation $\langle\,,\,\rangle$. Moreover, with a slight abuse of notation, given a translator $x:\Sigma^2\to\mathbb{R}^{3}$, we will use $x_i$ also to denote the restriction on $\Sigma$ of the $i$-th coordinate function $x_i$, that is we will identify $x_i(p)$ with $\langle x(p),e_i\rangle$, for any $p\in\Sigma$. 

The translator equation \eqref{Eq_Transl} turns out to be the Euler-Lagrange equation for the weighted area functional
\[
\mathrm{vol}_{f}(\Sigma)=\int_{\Sigma}e^{-f}d\mathrm{vol}_{\Sigma},
\]
when choosing $f=-x_3$. Clearly, the translator $\Sigma$ inherits the weighted structure of the ambient space $\mathbb{R}^{3}_{f}$. Indeed, we can consider intrinsically the weighted manifold
\[
 \Sigma^2_{\tilde{f}}=\left(\Sigma, x^{*}(\langle\,,\,\rangle), e^{-\tilde{f}}d\mathrm{vol}_{\Sigma}\right),
\]
where $\tilde{f}=f\circ x$. It is customary to drop the ``tilde'' in the weight function and to write $\Sigma_f^{2}$.

In the following lemma we recall a basic identity holding on translators that will be useful in what follows. For a proof we refer to \cite[Lemma 2.1]{IR_MathZ} (using there the opposite sign convention on the Laplacian).

\begin{lemma}\label{LemBasEq}
Set $f=-x_3$. Then with the $f$-Laplacian as in \eqref{Def:f-Laplacian}:
\[
\Delta_{f} x_i=\langle e_{i},e_{3}\rangle,\quad i=1,2,3.
\]
\end{lemma}

Since throughout the paper we will focus our attention on translators with finite entropy, it is useful to recall the definition of entropy and some related properties.

\begin{definition}[\cite{magman08}, \cite{CoMi}]
Let $\Sigma^n\subseteq \mathbb{R}^{n+k}$ be a submanifold. Given $x_0\in\mathbb{R}^{n+k}$ and $t_0>0$, set
\[
F_{x_0,t_0}(\Sigma)\doteq\frac{1}{(4\pi t_0)^{\frac n2}}\int_\Sigma e^{-\frac{\| x-x_0\|^2}{4 t_o}} d\mathrm{vol}_\Sigma.
\]
The \emph{entropy} $\lambda(\Sigma)$ of $\Sigma$ is then defined as
\[
\lambda(\Sigma)\doteq\sup_{x_0\in\mathbb{R}^{n+k},t_0>0} F_{x_0,t_0}(\Sigma).
\]
\end{definition}

Note that not all translators have finite entropy. It is well known (see \cite{Sun}) that having finite entropy is equivalent to have some control on the extrinsic volume growth. More specifically, a submanifold $\Sigma^n\subseteq\mathbb{R}^{n+k}$ has finite entropy if and only if
\[
\sup_{x\in\mathbb{R}^{n+k}, R>0}\frac{\mathrm{vol}\left(\Sigma\cap B^{\mathbb{R}^{n+k}}_R(x)\right)}{R^n}<+\infty,
\]
where $B^{\mathbb{R}^{n+k}}_R(x)$ is the standard ball in $\mathbb{R}^{n+k}$ of radius $R$ and centered at $x$.

As a consequence of the latter equivalence, we are able to relate the finite entropy assumption to potential-analytic properties of $2$-dimensional translators. More specifically, we have the validity of the following proposition. 

\begin{proposition}\label{Prop_parabolictranslators}
Let $\Sigma^2\subset \mathbb{R}^3$ be a translator. If $\Sigma$ has finite entropy then it is parabolic. Moreover, if $g\in C^{\infty}(\Sigma)$ is any function satisfying $\inf_\Sigma g>-\infty$, then $\Sigma$ is also $g$-parabolic.    
\end{proposition}
\begin{proof}
Fix a reference point $o\in\Sigma$. Since $\Sigma$ has finite entropy and extrinsic volumes are bigger than intrinsic volumes, it follows that
\[
\mathrm{vol}\left(B^\Sigma_R(o)\right)\leq \mathrm{vol}\left(\Sigma\cap B^{\mathbb{R}^{3}}_R(o)\right)\leq C\lambda(\Sigma) R^2.
\]
The parabolicity is now a direct consequence of Proposition \ref{Prop_quadparabolicity}. Finally, if $g\in C^{\infty}(\Sigma)$ is any function bounded from below, then we conclude by Lemma \ref{LemInfhPar} that $\Sigma$ is also $g$-parabolic. This latter fact can also be proven directly by observing that
\begin{align*}
    \mathrm{vol}_{g}(B_{R}^{\Sigma}(o))&=\int_{B_{R}^{\Sigma}(o)}e^{-g}d\mathrm{vol}_{\Sigma}\leq e^{-\inf_{\Sigma}g}\int_{B_{R}^{\Sigma}(o)}d\mathrm{vol}_{\Sigma}\\&=C\mathrm{vol}(B_{R}^{\Sigma}(o))\leq \tilde{C}R^{2}.
\end{align*}
with $\tilde{C}$ a constant depending on the entropy of $\Sigma$.
\end{proof}

\subsection{Stability for translators}\label{SubSect_Stab}
Translators being critical points for the weighted area functional it makes sense to investigate stability properties. Given a translator $x:\Sigma^m\to\mathbb{R}^{m+1}$ and a compactly supported normal variation $u\nu$, for some $u\in C_{c}^{\infty}(\Sigma)$, the second variation formula for the weighted area functional yields the quadratic form
\[
Q_{f}(u,u)=\int_{\Sigma}\left(|\nabla u|^2-|A|^2u^2\right)e^{-f}d\mathrm{vol}_{\Sigma},
\]
where $A$ is the second fundamental form of the immersion. In particular we say that the translator is $f$-stable if $Q_f(u,u)$ is non-negative for any $u\in C_{c}^{\infty}(\Sigma)$. 

An interesting fact, first observed by T. Ilmanen, is that $x:\Sigma^{m}\to\mathbb{R}^{m+1}$ is a translator if and only if it is minimal in $(\mathbb{R}^{m+1}, \widetilde{\langle\,,\,\rangle}\doteq e^{\frac{2x_{m+1}}{m}}\langle\,,\,\rangle)$. Therefore, letting $\tilde{g}=x^{*}\widetilde{\langle\,,\,\rangle}$, one could also consider stability of $x: (\Sigma^{m},  \tilde{g})\to(\mathbb{R}^{m+1}, \widetilde{\langle\,,\,\rangle})$ as a critical point of the area functional associated to this conformal metric. It turns out that these two notions of stability are equivalent; see \cite[Proposition 14]{CMZ15}.

\section{Rigidity in slabs}\label{Section_RigSlabs}

By Proposition \ref{Prop_parabolictranslators}, Theorem \ref{Thm_RigB} follows from the following formally more general result. While finite entropy thus implies parabolicity, the authors are in turn not aware of any examples of self-translating solitons which are parabolic but not of finite entropy.
\begin{theorem}\label{Thm:RigPar}
Let $\Sigma^2\subseteq \mathbb{R}^{3}$ be a translator which is parabolic and contained in a vertical slab of width $\pi\lambda$ for some $\lambda\in [1, \sqrt{2}]$. If $\Sigma\subseteq\left\{\Big(1-\frac{\lambda^2}{2}\Big)x_3+\alpha x_2\geq c\right\}$ for some $\alpha^2\geq\lambda^2-1$ and $c\in\mathbb{R}$, then necessarily $\lambda=1$, $\alpha=0$, and $\Sigma$ is a grim reaper cylinder $\Gamma\times\mathbb{R}$.
\end{theorem}

\begin{proof}
Let $x:\Sigma^{2}\to\mathbb{R}^{3}$ be a complete translator which is parabolic and contained in the slab $(0,\lambda\pi)\times \mathbb{R}^{2}$ and assume that $\inf_{\Sigma}\left(\Big(1-\frac{\lambda^2}{2}\Big)x_{3}+\alpha x_2\right)>-\infty$.

Consider the function $u:\Sigma\to\mathbb{R}$ given by
\begin{equation}\label{universal_super}
\ u=\frac{x_{3}}{\lambda^2}+\frac{\alpha}{\lambda^2}x_2+\ln\sin\left(\frac{x_1}{\lambda}\right).
\end{equation}
Note that the (tilted) grim reaper cylinders are level sets of $u$ in \eqref{universal_super} (depending on the parameters $\alpha,\lambda$). Hence, if we are able to prove that the function $u$ is constant, then we are done, by connectedness and completeness of $\Sigma$.

Since $x_{1}\in (0,\lambda\pi)$ we have that $\sin(\frac{x_1}{\lambda})>0$. Moreover
\[
\nabla u=\frac{\nabla x_{3}}{\lambda^2}+\frac{\alpha}{\lambda^2}\nabla x_2+\frac{1}{\lambda}\frac{\cos(x_1/\lambda)}{\sin(x_1/\lambda)}\nabla x_1
\]
and, as a consequence of Lemma \ref{LemBasEq}, for the $f$-Laplacian in \eqref{Def:f-Laplacian}:
\[
\Delta_f u=\frac{1}{\lambda^2}+\Delta_f \ln\sin\left(\frac{x_1}{\lambda}\right)=\frac{1}{\lambda^2}-\frac{1}{\lambda^2}\frac{|\nabla x_1|^2}{\sin^2(x_1/\lambda)}.
\]
Then
\begin{align*}
    \Delta_f e^u&=e^u(\Delta_f u+|\nabla u|^2)\\
    &=\frac{e^u}{\lambda^2}\left(1-\frac{|\nabla x_1|^2}{\sin^2(x_1/\lambda)}+\frac{|\nabla x_3|^2}{\lambda^2}+\frac{\cos^2(x_1/\lambda)}{\sin^2(x_1/\lambda)}|\nabla x_1|^2\right.\\
    &\quad\left. + \frac{\alpha^2}{\lambda^2}|\nabla x_2|^2+\frac{2}{\lambda}\frac{\cos(x_1/\lambda)}{\sin(x_1/\lambda)}\langle \nabla x_1,\nabla(x_3+\alpha x_2)\rangle+\frac{2\alpha}{\lambda^2}\langle \nabla x_3,\nabla x_2\rangle\right).
\end{align*}
On the other hand
\[
\ \frac{2}{\lambda}\frac{\cos(x_1/\lambda)}{\sin(x_1/\lambda)}\langle \nabla x_1,\nabla(x_3+\alpha x_2)\rangle=2\langle \nabla u, \nabla(x_3+\alpha x_2)\rangle-\frac{2}{\lambda^2}|\nabla (x_3+\alpha x_2)|^2.
\]
Hence we deduce that
\begin{align*}
\Delta_f e^u &=\frac{e^{u}}{\lambda^2}\left(1-|\nabla x_1|^2+\frac{|\nabla x_3|^2}{\lambda^2}+\frac{\alpha^2}{\lambda^2}|\nabla x_2|^2+\frac{2\alpha}{\lambda^2}\langle \nabla x_3,\nabla x_2\rangle\right.\\
&\quad\left. -\frac{2}{\lambda^2}|\nabla x_3|^2-\frac{4\alpha}{\lambda^2}\langle \nabla x_3,\nabla x_2\rangle-\frac{2\alpha^2}{\lambda^2}|\nabla x_2|^2\right)\\
&\quad+2\langle \nabla e^u,\nabla \frac{x_3+\alpha x_2}{\lambda^2}\rangle\\
&=\frac{e^{u}}{\lambda^2}\left(1-|\nabla x_1|^2-\frac{|\nabla x_3|^2}{\lambda^2}-\frac{\alpha^2}{\lambda^2}|\nabla x_2|^2-\frac{2\alpha}{\lambda^2}\langle \nabla x_3,\nabla x_2\rangle\right)\\
&\quad+2\langle \nabla e^u,\nabla \frac{x_3+\alpha x_2}{\lambda^2}\rangle.
\end{align*}
Set $g=\frac{2\alpha}{\lambda^2} x_2+\frac{2-\lambda^2}{\lambda^2}x_{3}$. Then, using the fact that $|\nabla x_i|^2=1-\langle e_i,\nu\rangle^2$, it is immediate to see that
\begin{align}
    \Delta_{g}e^u&=e^{g}\mathrm{div}(e^{-g}\nabla e^u)\label{Eq_Deltageu}\\
    &=\frac{e^u}{\lambda^4}\left((\lambda^2-1-\alpha^2)\langle e_{1}, \nu\rangle^2-(\alpha\langle e_3,\nu\rangle -\langle e_{2}, \nu\rangle)^{2}\right)\nonumber.
\end{align}
Note that by assumption $\inf_{\Sigma}g>-\infty$. Hence, Lemma \ref{LemInfhPar} implies that $\Sigma$ is $g$-parabolic and:
\begin{itemize}
\item \underline{\textbf{Case 1}}. If $\alpha^2>\lambda^2-1$ then $\Delta_{g}e^u\leq 0$ and $u$ has to be constant. The constancy of $u$ forces $\langle e_1,\nu\rangle=0$ and $\alpha\langle e_3,\nu\rangle=\langle e_2,\nu\rangle$. Since $H=\langle e_3,\nu\rangle$, using the identity
\[
1=\langle e_1, \nu\rangle^2+\langle e_2, \nu\rangle^2+\langle e_3, \nu\rangle^2,
\]
this would imply that $H$ should be a non-zero constant, which is absurd by the translator equation \eqref{Eq_Transl}.
\item \underline{\textbf{Case 2}}. Let us assume that $\alpha^2=\lambda^2-1$. Then, again $\Delta_{g}e^u\leq 0$ and $u$ has to be constant. In this case either $\lambda=1$ and we conclude that $\Sigma$ must be a grim reaper cylinder, or $\lambda\in (1,\sqrt2]$ (see the proof of Theorem \ref{Thm_RigB} below). However, in this case, the constancy of $u$ implies that $\Sigma$ is a tilted grim reaper cylinder. However any tilted grim reaper cylinder in our slab intersects the planes of the form $\frac{2\alpha}{\lambda^2} x_2+\frac{2-\lambda^2}{\lambda^2}x_{3}=c$.
\end{itemize}
\end{proof}

\begin{remark}
\rm{We remind the reader that the crucial property that $\Delta_{g}e^u\leq 0$, in the above proof, means that $e^u$, with $u$ as in \eqref{universal_super}, is a so-called universally $L$-superharmonic function, i.e. is superharmonic for this operator on any self-translating soliton.}
\end{remark}

\begin{proof}[Proof of Theorem \ref{Thm_RigB}]
We first focus on the case $\lambda\in[1,\sqrt{2}]$. In this case, as observed before, the assumption of finite entropy implies that $\Sigma$ is parabolic. We can thus conclude as desired by invoking Theorem \ref{Thm:RigPar}.

Concerning the case $\lambda>\sqrt{2}$, note that, as observed in \cite{C20}, the assumption that $\Sigma$ has finite entropy implies that it must be properly immersed. Furthermore, since $|\nabla f|$ is bounded, \cite[Example 1.14]{PRS_Memoirs} guarantees that the Omori-Yau maximum principle for the $f$-Laplacian holds on $\Sigma$. Specifically, for any $C^2$ function $w:\Sigma \to \mathbb{R}$ satisfying $\sup_\Sigma w = w^*<\infty$, there exists a sequence $\{x_n\} \subset \Sigma$ along which the following conditions hold: 
\begin{align*} 
&(i) \lim_{n \to +\infty} w(x_n)= w^*, \\ 
&(ii)\ |\nabla w(x_n)| < \frac{1}{n}, \\ 
&(iii)\ \Delta_f w(x_n) < \frac{1}{n}. 
\end{align*} 
Now, since we are assuming $\lambda > \sqrt{2}$, the condition $\Sigma \subseteq \left\{ \left( 1 - \frac{\lambda^2}{2} \right) x_3 + \alpha x_2 \geq c \right\}$ implies that the function $\left( \frac{\lambda^2}{2} - 1 \right) x_3 - \alpha x_2$ is bounded from above. We can then apply the Omori-Yau maximum principle to the function $w = \left( 1 - \frac{\lambda^2}{2} \right) x_3 + \alpha x_2$, obtaining
$$\frac1n>\Delta_f w(x_n)=\frac{\lambda^2}{2}-1>0.$$
Taking the limit as $n\to+\infty$ one immediately reaches a contradiction.

\end{proof}

\section{Rigidity of collapsed translators with prescribed boundary}\label{Section_RigPrescrBnd}

\begin{theorem}\label{Thm_RigStraightLine}
Let $\Sigma^2\subseteq \mathbb{R}^{3}$ be a translator with non-empty boundary which is (Neumann) parabolic and contained in a vertical slab of width $\omega$. If $\Sigma\subseteq \left\{x_{3}\leq a\right\}$ and $\partial\Sigma$ is contained in the straight line $\left\{x_{1}=c\right\}\cap\left\{x_{3}=a\right\}$ for some $a,c\in\mathbb{R}$. Then $\Sigma$ is a piece of a vertical plane parallel to the slab.
\end{theorem}

\begin{proof}
As a consequence of Lemma \ref{LemBasEq} and the slab assumption, the function $x_1$ satisfies:
\[
\begin{cases}
\Delta_f x_1=0,\\
\sup_\Sigma |x_1|<+\infty.
\end{cases}
\]
On the other hand, since, by assumption, $\inf_\Sigma f=\inf_\Sigma (-x_3)>-\infty$ and $\Sigma$ is parabolic, then $\Sigma$ is also $f$-parabolic, as a consequence of Lemma \ref{LemInfhPar}. As a consequence of Proposition \ref{PropAhlforsStokes} one then has
\[
\sup_{\Sigma} x_1=\sup_{\partial \Sigma} x_1=c,\quad \inf_\Sigma x_1=\inf_{\partial \Sigma} x_1=c
\]
In particular, $x_1\equiv c$ on $\Sigma$ and hence $\Sigma$ must be a piece of the plane $\left\{x_1=c\right\}$.
\end{proof}

Note that if $\Sigma$ is a complete translator of finite entropy, then, as a consequence of Proposition \ref{Prop_quadparabolicity}, it is parabolic for the surface Laplacian. Hence, the capacitary description of parabolicity (see e.g. Proposition \ref{PropAhlforsStokes} above) implies that, given a regular value $a$ of $x_3$, the sublevel set $\Sigma_{a}\doteq\{p\in\Sigma\,:\,x_3(p)\leq a\}$ is Neumann parabolic. Indeed, given any compact set $K\subset\Sigma_a$, every $\psi\in C_{c}^{\infty}(\Sigma)$ with $\psi\geq 1$ on $K$ gives rise to a test function $\varphi=\left.\psi\right|_{\Sigma_{a}}$ for $\mathrm{Cap}(K,\Sigma_a)$. Theorem \ref{Thm_RigPrescrSlice} follows now directly from Theorem \ref{Thm_RigStraightLine} . 

\appendix

\section{Equidistant surfaces to translators with respect to the conformal metric}\label{Equidist}
\subsection{Geodesics of $(\mathbb{R}^{3}, \widetilde{\langle\,,\,\rangle}$)}
Consider on $\mathbb{R}^{3}$ the conformal metric $\widetilde{\langle\,,\rangle}=e^{x_{3}}\langle\,,\,\rangle$. We recall that $\Sigma^{2}$ is a translator in $(\mathbb{R}^{3}, \langle\,,\,\rangle)$ if and only if it is minimal in $(\mathbb{R}^{3}, \widetilde{\langle\,,\,\rangle})$.
\medskip

The Christoffel symbols of $(\mathbb{R}^{3}, \widetilde{\langle\,,\,\rangle})$ are given by
\[
\ \Gamma_{ij}^{k}=\partial_{i}\frac{x_{3}}{2}\delta_{j}^{k}+\partial_{j}\frac{x_{3}}{2}\delta_{i}^{k}-\partial_{k}\frac{x_{3}}{2}\delta_{ij}.
\]
In particular, the non-vanishing ones are
\begin{align*}
    \begin{cases}
    \Gamma_{13}^{1}=\Gamma_{23}^{2}=\Gamma_{31}^{1}=\Gamma_{32}^{2}=\Gamma_{33}^{3}=\frac{1}{2},\\
    \Gamma_{11}^{3}=\Gamma_{22}^{3}=-\frac{1}{2}.
    \end{cases}
\end{align*}
A curve $\gamma=(\gamma_{1}, \gamma_{2}, \gamma_{3})$ in $\mathbb{R}^{3}$ is a geodesic in $(\mathbb{R}^{3}, \widetilde{\langle\,,\,\rangle})$ if and only if its component functions satisfy the geodesic equations
\begin{equation}\label{Eq_Geod}
 \begin{cases}
 \ddot\gamma_{1}+\dot{\gamma}_{1}\dot{\gamma}_{3}=0\\
 \ddot{\gamma}_{2}+\dot{\gamma}_{2}\dot{\gamma}_{3}=0\\
 \ddot{\gamma}_{3}+\frac{1}{2}\left\{(\dot{\gamma}_{3})^{2}-(\dot{\gamma}_{1})^{2}-(\dot{\gamma}_{2})^{2}\right\}=0.
 \end{cases}   
\end{equation}
Let $p\in \mathbb{R}^{3}$ and $v\in\mathbb{R}^{3}$ such that $\widetilde{\langle v, v\rangle}=1$ and let us denote by $\gamma$ the unique geodesic of $(\mathbb{R}^{3}, \widetilde{\langle\,,\,\rangle})$ satisfying $\gamma(0)=p$, $\dot{\gamma}(0)=v$. Then
\begin{equation}\label{Eq_dotGamma}
\dot{\gamma}_{1}^{2}+\dot{\gamma}_{2}^{2}=e^{-\gamma_{3}}-\dot{\gamma}_{3}^{2}.
\end{equation}
Hence the third equation in \eqref{Eq_Geod} reads as
\[
(e^{\gamma_{3}}\dot{\gamma}_{3})^\cdot=\frac12.
\]
Multiplying by $e^{\gamma_{3}}$ on both sides and integrating between $0$ and $t$, gives
\[
\ \gamma_{3}(t)=\ln\left(\frac{t^{2}+4\dot{\gamma}_{3}(0)e^{\gamma_{3}(0)}t+4e^{\gamma_{3}(0)}}{4}\right).
\]
Note that this expression is defined for all $t>0$.
\medskip

We now integrate the first and the second equation in \eqref{Eq_Geod}. For $i=1,2$ we have that
\begin{align*}
    \dot{\gamma}_{i}(t)=&\dot{\gamma}_{i}(0)e^{-\gamma_{3}(t)+\gamma_{3}(0)}=\frac{4\dot{\gamma}_{i}(0)e^{\gamma_{{3}(0)}}}{t^{2}+4\dot{\gamma}_{3}(0)e^{\gamma_{3}(0)}t+4e^{\gamma_{3}(0)}}\,\\
    \gamma_{i}(t)=&\gamma_{i}(0)+4\dot{\gamma}_{i}(0)e^{\gamma_{3}(0)}\int_{0}^{t}\frac{du}{u^2+4\dot{\gamma_{3}}(0)e^{\gamma_{3}(0)}u+4e^{\gamma_{3}(0)}}\,
\end{align*}
Note that
\begin{align*}
    &\int_{0}^{t}\frac{du}{u^2+4\dot{\gamma}_{3}(0)e^{\gamma_{3}(0)}u+4e^{\gamma_{3}(0)}}=\int_{0}^{t}\frac{du}{(u+2\dot{\gamma}_{3}(0)e^{\gamma_{3}(0)})^2+4e^{\gamma_{3}(0)}(1-\dot{\gamma}_{3}(0)^2e^{\gamma_{3}(0)})}\\
    =&\frac{e^{-\frac{\gamma_{3}(0)}{2}}}{2\sqrt{1-\dot{\gamma}_{3}(0)^2e^{\gamma_{3}(0)}}}\int_{0}^{t}du\left[\frac{1}{2e^{\frac{\gamma_{3}(0)}{2}}\sqrt{1-\dot{\gamma}_{3}(0)^{2}e^{\gamma_{3}(0)}}}\right]\left[\left(\frac{u+2\dot{\gamma}_{3}(0)e^{\gamma_{3}(0)}}{2e^{\frac{\gamma_{3}(0)}{2}}\sqrt{1-\dot{\gamma}_{3}(0)^2e^{\gamma_{3}(0)}}}\right)^2+1\right]^{-1}\\
    =&\frac{1}{2e^{\gamma_{3}(0)}\sqrt{\dot{\gamma}_{1}(0)^2+\dot{\gamma}_{2}(0)^2}}\left[\arctan\left(\frac{t+2\dot{\gamma}_{3}(0)e^{\gamma_{3}(0)}}{2e^{\gamma_{3}(0)}\sqrt{\dot{\gamma}_{1}(0)^2+\dot{\gamma}_{2}(0)^2}}\right)-\arctan\left(\frac{\dot{\gamma}_{3}(0)}{\sqrt{\dot{\gamma}_{1}^{2}(0)+\dot{\gamma}_{2}^{2}(0)}}\right)\right],
\end{align*}
where in the third identity we have used \eqref{Eq_dotGamma}. Hence, recalling that
\[
\ \arctan t_{1}-\arctan{t_{2}}=\arctan\left(\frac{t_{1}-t_{2}}{1+t_{1}t_{2}}\right),
\]
a simple computation yields
\[
\ \gamma_{i}(t)=\gamma_{i}(0)+2\frac{\dot{\gamma}_{i}(0)}{\sqrt{\dot{\gamma}_{1}(0)+\dot{\gamma}_{2}(0)}}\arctan\left(\frac{t\sqrt{\dot{\gamma}_{1}(0)+\dot{\gamma}_{2}(0)}}{t\dot{\gamma}_{3}(0)+2}\right).
\]
Summarising, the geodesic $\gamma$ can be parametrized by
\begin{align*}
    \gamma_{1}(t)=&\gamma_{1}(0)+2\frac{\dot{\gamma}_{1}(0)}{\sqrt{\dot{\gamma}_{1}(0)+\dot{\gamma}_{2}(0)}}\arctan\left(\frac{t\sqrt{\dot{\gamma}_{1}(0)+\dot{\gamma}_{2}(0)}}{t\dot{\gamma}_{3}(0)+2}\right)\\
    \gamma_{2}(t)=&\gamma_{2}(0)+2\frac{\dot{\gamma}_{2}(0)}{\sqrt{\dot{\gamma}_{1}(0)+\dot{\gamma}_{2}(0)}}\arctan\left(\frac{t\sqrt{\dot{\gamma}_{1}(0)+\dot{\gamma}_{2}(0)}}{t\dot{\gamma}_{3}(0)+2}\right)\\
    \gamma_{3}(t)=&\ln\left(\frac{t^{2}+4\dot{\gamma}_{3}(0)e^{\gamma_{3}(0)}t+4e^{\gamma_{3}(0)}}{4}\right).
\end{align*}

\subsection{$\widetilde{\langle\,,\,\rangle}$-equidistant surfaces to translators} Let $\Sigma^2$ be a translator in $\mathbb{R}^{3}$, $p=(p_{1}, p_{2}, p_{3})\in \Sigma^2$ and denote by $\nu=(\nu_{1}, \nu_{2}, \nu_{3})$ its unit normal. Consider the geodesic $\gamma:\mathbb{R}\to(\mathbb{R}^{3}, \widetilde{\langle\,,\,\rangle})$ with initial conditions $\gamma(0)=p$ and $\dot{\gamma}(0)=e^{-p_{3}/2}\nu$. By the translator equation \eqref{Eq_Transl} we have that 
\[
\ \dot{\gamma}_{3}(0)=H(p)e^{-p_{3}/2},
\]
where $H(p)$ denotes the mean curvature of $\Sigma$ evaluated at the point $p$.
Therefore
\[
\ \dot{\gamma}_{1}(0)^{2}+\dot{\gamma}_{2}(0)^{2}=e^{-p_{3}}(1-H(p)^2),
\]
so that
\[
\ \frac{\dot{\gamma}_{1}(0)}{\sqrt{\dot{\gamma}_{1}(0)^2+\dot{\gamma}_{2}(0)^{2}}}=\frac{\nu_{1}}{\sqrt{1-H(p)^2}}.
\]
Substituting in \eqref{Eq_Geod}, we therefore obtain that the surface at distance $t$ from $\Sigma$ in the direction of $\dot{\gamma}(0)=e^{-p_{3}/2}\nu$ is parametrized by
\begin{equation}\label{Eq_EquidistToTransl}
\begin{cases}
x_{1,t}=p_{1}+\frac{2\nu_{1}}{\sqrt{1-H^2}}\arctan\left(\frac{t\sqrt{1-H^2}}{tH+2e^{p_3/2}}\right)\\
x_{2,t}=p_{2}+\frac{2\nu_{2}}{\sqrt{1-H^2}}\arctan\left(\frac{t\sqrt{1-H^2}}{tH+2e^{p_3/2}}\right)\\
x_{3,t}=\ln\left(\frac{t^{2}+4He^{p_{3}/2}t+4e^{p_{3}}}{4}\right).
\end{cases}
\end{equation}

To analyze a particular situation in detail, we will now focus on the simpler special case of $H\equiv 0$, considering the plane given by
\[
\Sigma=\Sigma_{0}\doteq\left\{(0, x_{2}, x_{3})\in\mathbb{R}^{3}: x_{2}, x_{3}\in\mathbb{R}\right\}.
\]
By \eqref{Eq_EquidistToTransl}, for any $t\in\mathbb{R}$ the equidistant surface $\Sigma_{t}\subseteq\mathbb{R}^{3}$ at (signed) distance $t$ from $\Sigma_{0}$ is parametrized by
\begin{equation*}
    \begin{cases}
    x_{1,t}=2\arctan\left(\frac{te^{-x_{3}/2}}{2}\right)\\
    x_{2,t}=x_2\\    x_{3,t}=\ln\left(\frac{t^{2}+4e^{x_{3}}}{4}\right).
    \end{cases}
\end{equation*}
\medskip

Let us study some features of the family of equidistant surfaces $\left\{\Sigma_{t}\right\}_{t\in\mathbb{R}}$:
\begin{itemize}
\item[(i)]$\left\{\Sigma_{t}\right\}_{t\in\mathbb{R}}$ foliate $(-\pi,\pi)\times\mathbb{R}\times\mathbb{R}$.
\item[(ii)] Let us focus here on the case $t>0$. Setting $x_1=2\arctan\left(\frac{te^{-x_{3}/2}}{2}\right)$ we can reparametrize $\Sigma_{t}$ as follows:
\[
\ \left\{\left(x_1,x_2,\ln\left(\frac{t^{2}}{4}\right)-2\ln\sin\left(\frac{x_1}{2}\right)\right): (x_1,x_2)\in(0,\pi)\times\mathbb{R}\right\}.
\]
Therefore the unit normal to $\Sigma_{t}$ is
\[
\ \nu_{t}=\left(\cos\left(\frac{x_1}{2}\right), 0, \sin\left(\frac{x_1}{2}\right)\right)
\]
and the first and the second fundamental form of $\Sigma_{t}$ are given by
\[
\ g_{\Sigma_{t}}=\left(\begin{array}{cc}\frac{1}{\sin^{2}(x_1/2)}&0\\0&1\end{array}\right),\qquad II_{\Sigma_{t}}=\left(\begin{array}{cc}\sin^{2}(x_1/2)&0\\0&1\end{array}\right)\left(\begin{array}{cc}\frac{1}{2\sin(x/2)}&0\\0&0\end{array}\right).
\]
In particular,
\[
\ H_{\Sigma_{t}}=\frac{1}{2}\sin(x_1/2)=\frac{1}{2}\langle e_{3}, \nu_{\Sigma_{t}}\rangle.
\]
Note that any $\Sigma_{t}$ looks like an open piece of a rescaled (with a factor $2$) grim reaper cylinder. In particular $\Sigma_{t}$ is not a unit-speed translator.
\end{itemize}


\end{document}